\documentclass[a4paper,10pt,notitlepage,reqno]{article}
\usepackage[english]{babel}
\usepackage[latin1]{inputenc}
\usepackage[T1]{fontenc}
\usepackage{amssymb,amsthm,amsmath} 
\usepackage{mathtools}
\usepackage{mathrsfs}
\usepackage{enumerate}
\usepackage{authblk}
\usepackage{color}
\usepackage{geometry}
\usepackage{hyperref}
\newcommand{\NN}{\mathbb{N}}
\newcommand{\RR}{{\mathbb{R}}}
\newcommand{\CC}{{\mathbb{C}}}

\newcommand{\TT}{{\mathbb{T}}}

\newcommand{\cA}{{\mathcal A}}

\newcommand{\cF}{{\mathcal{F}}}
\newcommand{\cH}{{\mathcal{H}}}

\newcommand{\Dom}{{\operatorname{Dom}}}

\newcommand{\cf}{\emph{cf.}}

\newcommand{\eps}{\varepsilon}

\newcommand{\ov}{\overline}

\newcommand\la{\lambda}

\renewcommand{\d}{{\textrm{d}}}

\renewcommand\eps{\varepsilon}

\newcommand\sd{\sigma_{\rm disc}}
\newcommand\sess{\sigma_{\rm ess}}

\newcommand\mydot{\,\cdot\,}
\newcommand\ds{\displaystyle}

\newcommand\wh{\widehat}

\theoremstyle{plain} 
\newtheorem{theorem}{Theorem}[section]
\newtheorem{lemma}{Lemma}[section]

\newtheorem{corollary}{Corollary}[section]

\theoremstyle{definition}

\theoremstyle{remark}
\newtheorem{remark}{Remark}[section]

\newcommand{\normeq}[1]{
	{\left\vert\kern-0.25ex\left\vert\kern-0.25ex\left\vert #1 
	\right\vert\kern-0.25ex\right\vert\kern-0.25ex\right\vert}}

{\left\lbrace\begin{array}{r@{\hspace{1mm}}ll}}%
	{\end{array}\right.}

\title{\textbf{Spectral analysis of the spin-boson Hamiltonian 
		\\with two bosons for arbitrary coupling and \\bounded 
		dispersion relation}}

\author{Orif \,O.\ Ibrogimov} 
\affil{Department of Mathematics, Faculty of Nuclear Sciences 
	and Physical Engineering, Czech Technical University,
	Trojanova 13, 12000 Prague 2, Czech Republic; 
	ibrogori@fjfi.cvut.cz}
\begin{document}
\date{\small 29 January 2019}
\maketitle
\begin{abstract}
\noindent
We study the spectrum of the spin-boson Hamiltonian with 
two bosons for arbitrary coupling $\alpha\!>\!0$~in the case 
when the dispersion relation of the free field is a bounded 
function. We derive an explicit description of the essential 
spectrum which consists of the so-called two and 
three-particle branches that can be separated by a gap if 
the coupling is sufficiently large. It turns out, that 
depending on the location of the coupling constant and the 
energy level of the atom (with respect to certain constants
depending on the maximal and the minimal values of the boson 
energy) as well as the validity or the 
violation of the infrared regularity type conditions, the 
essential spectrum is either a single finite interval or a 
disjoint union of at most six finite intervals. 
The corresponding critical values of the coupling constant 
are determined explicitly and the asymptotic lengths of 
the possible gaps are given when $\alpha$ 
approaches to the respective critical value. Under minimal 
smoothness and regularity conditions on the boson 
dispersion relation and the coupling function, we show that 
discrete eigenvalues can never accumulate at the edges of 
the two-particle branch. Moreover, we show the absence of 
the discrete eigenvalue accumulation at the edges of the 
three-particle branch in the infrared regular case.
\end{abstract}
\footnotetext{\emph{Keywords}. Spin-boson Hamiltonian, 
Fock space, discrete/essential spectrum, 
Schur complement, 
Weyl inequality.}
\footnotetext{\emph{2010 Mathematics Subject Classification}. 
81Q10, 47A10, 70F07, 47G10.}
%
%
%
\section{Introduction}\label{sec:intro}
%
The spin-boson Hamiltonian is the energy operator of a quantum 
mechanical model describing a two-level system which is linearly 
coupled to the quantized field of bosons. It is formally given 
by the expression
\begin{equation}\label{expression:full.sb.ham}
	H=\eps\sigma_z\otimes\mathbb{I}+\mathbb{I}\otimes
	\int_{\RR^d}\omega(k)a^*(k)a(k)\,\d k + 
	\alpha\sigma_x\otimes\int_{\RR^d}(\la(k)a^*(k)
		+\ov{\la(k)}a(k))\,\d k
\end{equation}
and acts on the Hilbert space 
\begin{equation}
	\mathscr{H}=\CC^2\otimes\cF_s\bigl(L^2(\RR^d)\bigr),
\end{equation}  
where $\cF_s\bigl(L^2(\RR^d)\bigr)$ is the symmetric 
(bosonic) Fock space over $\RR^d$. In 
\eqref{expression:full.sb.ham}, $a^*(k)$ and $a(k)$ are the 
boson creation and annihilation operators, 
\begin{equation*}
\sigma_z=\begin{pmatrix}
1 & 0\\
0 & -1
\end{pmatrix}, 
\quad 
\sigma_x=\begin{pmatrix}
0 & 1\\
1 & 0
\end{pmatrix}
\end{equation*}
are the Pauli (spin) matrices, the constants $\eps>0$ and $-\eps$ 
are the energy levels of the atom corresponding to its excited and 
ground states, respectively, $\omega(k)$ is the boson dispersion, 
$\alpha>0$ is the coupling constant and $\la(k)$ is the coupling 
function.

For the photon case, spectral 
and scattering properties of the full spin-boson Hamiltonian as 
well as of its finite photon approximations have been investigated 
extensively and as a by-product sophisticated techniques have been 
developed. The corresponding literature is enormous and we limit 
ourselves to citing \cite{Spohn-CMP-1989,Zhukov-Milnos-1995, 
	Huebner-Spohn-1995, Minlos-Spohn-1996, Gerard-1996,  
	Arai-Hirokawa-1997, deMonvel-Sahbani-1998, Skibsted-1998, 
	Bach-Froechlich-Sigal-Adv.Math1998, Arai-2000, DerJak-JFA2001, 
	Hirokawa-2001RMP, Galtbayar-Jensen-Yajima-2003, 
	Angelescu-Minlos-Ruiz-Zagrebnov-2008, 	Abdesselam-2011, 
	Hasler-Herbst-2011, DeRoeck-Griesemer-Kupianinen-2015, 
	Merkli-CMP2015, 
	Bach-Ballesteros-Koenenberg-Menrath-2017, 
	Braunlich-Hasler-Lange-2018} (and the related work 
	\cite{Moeller-2005, Miyao-2009, Moller-Rasmussen-2013, 
		MNR-2016-1D}).
In these studies, the dispersion of the free field is taken 
to be either the relativistic dispersion 
$\omega(k)=\sqrt{k^2+m^2}$ or its limiting cases $\omega(k)=|k|$, 
$\omega(k)=\frac{k^2}{2m}$, and sometimes as a general unbounded 
and almost everywhere continuous function preserving all the 
features of these physical photon dispersion relations. 

In general, the dispersion relation $\omega\colon\RR^d\to[0,\infty)$ 
and the coupling function $\la\colon\RR^d\to\CC$ are fixed by the 
physics of the problem. In view of different applications of 
the spin-boson Hamiltonian, one likes 
to consider them as free parameter functions and impose 
only some general conditions such as 
\begin{equation}\label{cond:la.la.omega.in.L2}
	\la\in L^2(\RR^d), \quad 
		\frac{\la}{\sqrt{\omega-\omega_0}}\in L^2(\RR^d),
\end{equation}
where
\begin{equation}\label{def:omega0}
	\omega_0:=\inf_{k\in\RR^d}\omega(k).
\end{equation}
We note that the first condition in \eqref{cond:la.la.omega.in.L2} 
is the minimal assumption on the coupling function from the spectral 
theory viewpoint as it is needed to guarantee
the closability of the energy operator. To the best of our knowledge, 
almost every study on the spectrum of the spin-boson model in 
the up-to-date literature assumes at least \eqref{cond:la.la.omega.in.L2} 
or its strengthened version where the second condition in 
\eqref{cond:la.la.omega.in.L2} is replaced by  
\begin{equation}\label{cond2:la.omega.in.L2}
	\frac{\la}{\omega-\omega_0}\in L^2(\RR^d),
\end{equation}
which is known as the \emph{infrared regularity condition} 
(\cf~\cite{Hirokawa-2001RMP, Davies1981_SymmetryBreaking}). 
The recent work \cite{Ibrogimov-AHP2018} studies the problem 
of the explicit description of the essential spectrum and the 
finiteness of the discrete spectrum for the spin-boson model with 
two photons \emph{for arbitrary coupling} where the only requirement 
on the coupling function was its square integrability, whereas the 
photon dispersion relation was assumed to be a continuous 
(almost everywhere) and \emph{unbounded} function.

Motivated by solid state physics applications of the spin-boson 
Hamiltonian\footnote{We note that one has $\omega(k)\equiv1$ for 
	the (original) Fr\"ohlich polaron, \cf~\cite{H.Froehlich1954, 
		Spohn1988:large.polaron}} and related 
	finite volume approximation problems (\cf~\cite{Arai-Hirokawa-1997}),
in this paper we undertake the spectral analysis of the spin-boson 
Hamiltonian with two bosons for \emph{arbitrary coupling} and 
\emph{bounded} dispersion relation. The corresponding truncated 
spin-boson Hamiltonian is obtained from the full spin-boson 
Hamiltonian by the compression onto the subspace of two bosons which is 
given by the tensor product of $\CC^2$ and the truncated Fock space
\begin{equation}
\cF_s^2:=\CC \oplus L^2(\RR^d) \oplus L^2_s(\RR^d\times\RR^d).
\end{equation}
Here $L^2_s(\RR^d\times\RR^d)$ stands for the subspace of the 
Hilbert space $L^2(\RR^d\times\RR^d)$ consisting of symmetric 
functions and equipped with 
the $1/2$-inner product of the latter. For 
$f=\bigl(f^{(\sigma)}_0, f^{(\sigma)}_1, f^{(\sigma)}_2\bigr)\in 
	\CC^2\otimes\cF_s^2$, where $\sigma=\pm$ is a discrete variable, 
the Hamiltonian of our system is given by the formal expression 
\begin{equation}\label{Hamiltonian}
\begin{aligned}
(H_{\alpha}f)^{(\sigma)}_0 &= \sigma\eps f^{(\sigma)}_0 + 
	\alpha\int_{\RR^d} \la(q)f^{(-\sigma)}_1\!(q) \,\d q,\\
(H_{\alpha}f)^{(\sigma)}_1\!(k) &=  (\sigma\eps+\omega(k))f^{(\sigma)}_1\!(k)+
	\alpha\la(k)f^{(-\sigma)}_0+
		\alpha\int_{\RR^d}f^{(-\sigma)}_2\!(k,q)\la(q) \,\d q,\\
(H_{\alpha}f)^{(\sigma)}_2\!(k_1,k_2) &= (\sigma\eps+\omega(k_1)+
	\omega(k_2))f^{(\sigma)}_2\!(k_1,k_2)+
	\alpha\la(k_1)f^{(-\sigma)}_1\!(k_2)+
		\alpha\la(k_2)f^{(-\sigma)}_1\!(k_1). 
\end{aligned}
\end{equation}
Throughout the paper
we assume 
that the parameter $\eps>0$ is fixed, the dispersion 
relation $\omega\geq0$ is a non-constant, bounded, continuous 
function (although almost everywhere continuity will be enough 
in our analysis) and the coupling function $\la$ is not 
identically zero 
and $\la\in L^2(\RR^d)$. 
If 
$\la$ 
is identically zero on $\RR^d$, 
then the bosons do not couple to the atom and the description 
of the spectrum becomes straightforward. The case of constant 
dispersion relation was considered in 
\cite{Ibrogimov-Froehlich.polaron.2018} where it was possible 
to make the spectral analysis very explicit. The spatial dimension, 
$d\geq1$, plays no particular role in our analysis and is left 
arbitrary. We make the notation
\begin{equation}\label{def:omega1}
\omega_1:=\sup_{k\in\RR^d}\omega(k).
\end{equation}

The goal of this paper is to give an explicit description of 
the essential spectrum, analyse its structure and study the 
finiteness of the discrete eigenvalues in the gaps as well as 
outside of the essential spectrum for all values of the coupling 
constant. Our strategy is based on reducing the problem to the 
spectral analysis of 
a family of self-adjoint $2\times2$ operator matrices 
\begin{equation}\label{reduced.op.mat.}
\cA:=\begin{pmatrix}
A & B \\[1ex]
C & D
\end{pmatrix}
\end{equation}
depending on two parameters $\gamma\in\RR$ and $\alpha>0$ which 
act on the Hilbert space 
\begin{equation}
\cH:=L^2(\RR^d) \oplus L^2_s(\RR^d\times\RR^d).
\end{equation} 
For $f\in L^2(\RR^d)$ and $g\in L^2_s(\RR^d\times\RR^d)$, the operator 
entries of \eqref{reduced.op.mat.} are formally defined by the relations
\begin{equation}\label{op.entr.op.mat.A}
\begin{aligned}
&(Af)(k) = (-\gamma+\omega(k))f(k), 
\quad 
&&(Bg)(k)= \alpha\int g(k,q)\la(q)\,\d q,
\\
&(Cf)(k,q)=\alpha\ov{\la(k)}f(q)+\alpha\ov{\la(q)}f(k), 
\quad 
&&(Dg)(k,q)= (\gamma+\omega(k)+\omega(q))g(k,q). 
\end{aligned}
\end{equation}
The methods we employ to achieve the aforementioned results are direct and 
quite simple. We particularly benefit from block operator matrix techniques 
involving Schur complements and the corresponding Frobenius-Schur 
factorizations combined with the standard perturbation theory.  
It turns out that the essential spectrum of the reduced operator matrix $\cA$ 
consists of the union of the interval $[2\omega_0+\gamma, 2\omega_1+\gamma]$ 
(i.e.~the spectrum of~$D$) with an additional branch corresponding to 
the essential spectrum of a non-linear pencil of multiplication operators 
in $L^2(\RR^d)$ which represents the Schur complement of $D-z$ in the 
operator matrix $\cA-z$ in the Calkin algebra (Theorem~\ref{thm:ess.spec.A}). 
Depending on the location of the parameters $\alpha>0$ and $\eps$ 
(with respect to certain constants
formed out of $\omega_0$ and $\omega_1$) as well as the validity or the 
violation of the second condition in~\eqref{cond:la.la.omega.in.L2} and 
its natural counterpart
\begin{equation}
\frac{\la}{\sqrt{\omega_1-\omega}}\in L^2(\RR^d),
\end{equation}
this additional branch consists of a single finite interval which 
can be in either side of $[2\omega_0+\gamma, 2\omega_1+\gamma]$, or 
two finite intervals, one in each side of the latter. The corresponding 
critical values of the coupling constant, the transitions at which open a 
new gap in the essential spectrum, are found explicitly and the asymptotic 
lengths of the possible gaps are determined when $\alpha$ approaches to 
the respective critical value (see~Theorem~\ref{thm:asymp.length.og.gaps}). 
Inspired by the splitting trick we have developed in the recent work 
\cite{Ibrogimov-AHP2018}, it is shown that no edge of the essential 
spectrum other than $2\omega_0+\gamma$ and $2\omega_1+\gamma$ can be an 
accumulation point of the discrete spectrum (Theorem~\ref{thm:disc.spec.1}). 
The absence of the discrete eigenvalue accumulation at the points 
$2\omega_0+\gamma$ and $2\omega_1+\gamma$ can also be guaranteed 
whenever the infrared regularity condition \eqref{infrared.reg.conditions} 
and its natural counterpart
\begin{equation}
\frac{\la}{\omega_1-\omega}\in L^2(\RR^d)
\end{equation}
(or their weaker version, in the case when these points are not the top or 
the bottom of the essential spectrum) are satisfied 
(Theorem~\ref{thm:disc.spec.2}). Finally, in 
Theorem~\ref{thm:spin-boson-spec}, we deduce the spectral properties of the 
Hamiltonian~\eqref{Hamiltonian} from the spectral 
information for the operator matrix $\cA$ with $\gamma=\eps$ and 
$\gamma=-\eps$.

In spite of being self-adjoint and bounded, the operator matrix 
$\cA$ in \eqref{op.entr.op.mat.A} is, up to our knowledge, not 
covered by any of the currently existing abstract results such as 
\cite{ALMS94, AMS-1998, MR2363469, Kraus-Langer-Tretter-2004, 
	Langer-Langer-Tretter-2002} and requires an individual analysis. 
This is mainly due to the non-compactness of partial-integral 
operators and the fact that both diagonal entries have empty 
discrete spectrum. Although in our setting the underlying domain 
is the whole Euclidean space $\RR^d$, our results can be 
trivially translated into the setting with bounded underlying 
domains $\Omega\subset\RR^d$ such as the torus $\TT^d$. 
To achieve this, it suffices to replace the parameter functions 
$\omega$ and $\la$ with their cut-offs $\chi_{\Omega}\omega$ 
and $\chi_{\Omega}\la$. In this sense our results generalize 
and clarify the main results of \cite{MNR-2016-1D, Rasulov-2016TMF} 
in several respects. The splitting trick that we employ to show the 
absence of the 
eigenvalue accumulation at the edges of the two-particle branch 
(Theorem~\ref{thm:disc.spec.1}) is a modification of the similar 
trick developed in our recent work \cite{Ibrogimov-AHP2018}. 
Though inspired by it, our splitting trick is quite different from 
the well known splitting trick of  
Schr\"odinger operator theory (\cf~\cite{Seto-1974, Klaus-1977, 
	Newton-1983, Duclos-Exner-1995, Abdullaev-Ikromov-2007, 
	Ikromov-Sharipov-FAA-1998}). It is based on a 
``special decomposition'' of a ``special portion'' of the 
Birman-Schwinger type kernel. Unlike our approach, the 
classical splitting trick in our setting would require some 
additional regularity on the parameter functions. 

Throughout the paper we adopt the following notations. For a 
self-adjoint operator $T$ acting in a Hilbert space and an 
interval $(a, b)\subset\RR\setminus\sess(T)$, we denote by 
$N((a,b); T)$ the dimension of the spectral subspace of $T$ 
corresponding to the interval $(a,b)$ which coincides with 
the number of the discrete eigenvalues of $T$ in the interval 
$(a,b)$ (counted with multiplicities). The integrals with no 
indication of the limits imply the integration over the whole
space $\RR^d$ or $\RR^d\times\RR^d$. 
%
%
\section{Spectral properties of the reduced operator matrix}
%
\subsection{Preliminaries}
%
Since $\la\in L^2(\RR^d)$, it follows that 
$B\colon L^2_s(\RR^d\times\RR^d)\to L^2(\RR^d)$ is an everywhere 
defined bounded operator with the adjoint $B^*=C$. Hence, the operator 
matrix $\cA:\cH\to\cH$ is everywhere defined and self-adjoint 
(\cf~\cite[Theorem~V.4.3]{Kato}).
Note that 
$D:L^2_s(\RR^d\times\RR^d)\to L^2_s(\RR^d\times\RR^d)$
is an everywhere defined self-adjoint operator with the 
spectrum
\begin{equation}
\sigma(D)=[2\omega_0+\gamma, 2\omega_1+\gamma],
\end{equation}  
where $\omega_0$ and $\omega_1$ are defined by 
\eqref{def:omega0} and \eqref{def:omega1}.

As it was mentioned in Introduction, our approach to study 
the spectrum of the $2\times2$ operator matrix $\cA$ is 
based on the so-called Schur complements and the 
corresponding Frobenius-Schur factorization, 
\cf~\cite{Tre08}. For $z\in\rho(D)$, the Schur complement 
of $D-z$ in the operator matrix $\cA-z$ is given by  
\begin{equation*}
	S(z) := A-z-B(D-z)^{-1}B^*, \quad 
	\Dom(S(z)) := L^2(\RR).
\end{equation*}
Simple calculations yield the representation
\begin{equation}\label{Schur2}
	S(z) = \Delta(z)-K(z), 
\end{equation}
where $\Delta(z)\colon L^2(\RR^d)\to L^2(\RR^d)$ is the maximal 
operator of multiplication by the function 
\begin{equation}\label{def:Delta}
	\Delta(k;z):=\omega(k)-\gamma-z-\alpha^2\int\frac{
	|\la(q)|^2\,\d q}{\omega(k)+\omega(q)+\gamma-z},
\end{equation}
and $K(z)\colon L^2(\RR^d)\to L^2(\RR^d)$ is everywhere defined 
integral operator with the kernel 
\begin{equation}\label{kernel:p}
	p(k,q;z):=\frac{\alpha^2\ov{\la(k)}\la(q)}
		{\omega(k)+\omega(q)+\gamma-z}.
\end{equation}
Since $B(D-z)^{-1}\colon L^2_s(\RR^d\times\RR^d)\to L^2(\RR^d)$ is 
an everywhere defined bounded operator, the following 
Frobenius-Schur factorization holds 
\begin{equation}\label{Frobenius-Schur}
	\cA-z = \begin{pmatrix}
	I & \! B(D-z)^{-1}\\[1ex]
	0 & \! I
	\end{pmatrix}
	\begin{pmatrix}
	S(z) & \! 0\\[1ex]
	0 & \! D-z
	\end{pmatrix}
	\begin{pmatrix}
	I & \! 0\\[1ex]
	(D-z)^{-1}B^* & \! I
	\end{pmatrix}.
\end{equation}
%
In view this factorization, we can reduce the study of 
the essential spectrum of the operator matrix $\cA$ to 
that of the non-linear pencil $S$. More precisely, 
we have
\begin{equation}\label{Schur.power.ess.spec}
	\sess(\cA)\cap\rho(D)=\sess(S)\cap\rho(D).
\end{equation}
Moreover, we have the following crucial result the proof 
of which is similar to that of~ 
\cite[Lemma~3]{Ibrogimov-AHP2018} and follows from 
the 
factorization \eqref{Frobenius-Schur}.
\begin{lemma}\label{lem:EV.count.1}
If $z\in\rho(D)$ satisfies 
$N\bigl((-\infty,0);S(z)\bigr)<\infty$, 
then there exists $\delta>0$ such that 
	\begin{equation}
		\sd(\cA)\cap(z-\delta,z)=\varnothing.
	\end{equation}
Similarly, if $z\in\rho(D)$ satisfies 
$N\bigl((-\infty,0);-S(z)\bigr)<\infty$, then there 
exists $\delta>0$ such that 
	\begin{equation}
		\sd(\cA)\cap(z, z+\delta)=\varnothing.
	\end{equation}
\end{lemma}
\begin{remark}\label{rem:on.Delta.and.K}
For each $z\in\rho(D)$, the operators 
$S(z):L^2(\RR^d)\to L^2(\RR^d)$ and 
$\Delta(z):L^2(\RR^d)\to L^2(\RR^d)$ are bounded and 
self-adjoint on $\cH_1$. Moreover, the integral operator 
$K(z):L^2(\RR^d)\to L^2(\RR^d)$ is self-adjoint and of 
Hilbert-Schmidt class since its kernel 
$p(\mydot,\mydot;z)$ satisfies
\begin{equation}
	\|p(\mydot,\mydot;z)\|_{L^2_s(\RR^d\times\RR^d)} \leq 
	\frac{2\alpha^2\|\la\|^2_{L^2(\RR^d)}}
	{\text{dist}(z,\sigma(D))}<\infty.
\end{equation} 
\end{remark}
\noindent
The rest of this subsection is devoted to the analysis of 
possible zeros of the continuous function
\begin{equation}\label{Phi}
	\Phi(z)=-\gamma-z-
	\alpha^2\int\frac{|\la(q)|^2\,\d q}{\omega(q)+\gamma-z}, 
	\quad 
	z\in(-\infty,\omega_0+\gamma)
	\cup(\omega_1+\gamma, \infty). 
\end{equation}
It is easy to see that $\Phi$ is strictly decreasing and 
\begin{equation}\label{lim.at.pminfty}
	\lim_{z\downarrow-\infty}\Phi(z)=+\infty, \quad 
	\lim_{z\uparrow +\infty}\Phi(z)=-\infty.
\end{equation}
Moreover, the monotone convergence theorem guarantees 
the existence of the (possibly improper) limits
\begin{equation}\label{mon.conv.thm:lim.at.boundary}
	\lim_{z\uparrow \omega_0+\gamma}\Phi(z)=:
			\Phi(\omega_0+\gamma), \quad 
	\lim_{z\downarrow \omega_1+\gamma}\Phi(z)=:
	\Phi(\omega_1+\gamma).
\end{equation}
If the condition 
\begin{equation}\label{strong.ultrared.regularity.m}
	\frac{\la}{\sqrt{\omega-\omega_0}}\in L^2(\RR^d)
\end{equation} 
is satisfied, then the first limit in 
\eqref{mon.conv.thm:lim.at.boundary} is finite for all 
$\alpha>0$ and is given by 
\begin{equation}
	\Phi(\omega_0+\gamma):=-2\gamma-\omega_0-
	\alpha^2\int\frac{|\la(q)|^2\,\d q}{\omega(q)-\omega_0}.
\end{equation}
We distinguish the two cases:
\begin{itemize}
\item \emph{Either $\gamma\geq-\frac{\omega_0}{2}$ and 
	$\alpha>0$ is arbitrary, or $\gamma<-\frac{\omega_0}{2}$ 
	and $\alpha>\alpha_1$, where}
	\begin{equation}\label{alpha.cr.1}
		\ds\alpha_1:=\alpha_1(\gamma):=
		\frac{\sqrt{-2\gamma-\omega_0}}{\bigl\|
			\frac{\la}{\sqrt{\omega-\omega_0}}
			\bigr\|_{L^2(\RR^d)}}.
	\end{equation}
	In this case, we have $\Phi(\omega_0+\gamma)<0$. 
	Hence, the first relation in \eqref{lim.at.pminfty} 
	and the monotonicity imply that the continuous function 
	$\Phi$ has \emph{a unique zero} -- denoted by 
	$E:=E(\alpha,\gamma)$ -- in the interval 
	$(-\infty, \omega_0+\gamma)$.
\item \emph{The case $\gamma<-\frac{\omega_0}{2}$ and 
	$0<\alpha\leq \alpha_1$}. In this case, we have 
	$\Phi(\omega_0+\gamma)\geq0$. Hence, the first relation 
	in \eqref{lim.at.pminfty} and the monotonicity imply that 
	the continuous function $\Phi$ does not vanish in the 
	interval $(-\infty, \omega_0+\gamma)$.
\end{itemize}
If \eqref{strong.ultrared.regularity.m} does 
not hold, then the first limit in 
\eqref{mon.conv.thm:lim.at.boundary} is negative infinity 
for all $\alpha>0$. Hence, the first relation in 
\eqref{lim.at.pminfty} and the monotonicity imply that 
the continuous function $\Phi$ has \emph{a unique zero} 
(again denoted by) $E:=E(\gamma, \alpha)$ in the interval 
$(-\infty, \omega_0+\gamma)$ for all $\alpha>0$.

\smallskip
\noindent
Similarly, if the condition 
\begin{equation}\label{strong.ultrared.regularity.M}
	\frac{\la}{\sqrt{\omega_1-\omega}}\in L^2(\RR^d)
\end{equation} 
is satisfied, then the second limit in 
\eqref{mon.conv.thm:lim.at.boundary} is finite for all 
$\alpha>0$ and is given by 
\begin{equation}
	\Phi(\omega_1+\gamma):=
	-2\gamma-\omega_1+\alpha^2
	\int\frac{|\la(q)|^2\,\d q}{\omega_1-\omega(q)}.
\end{equation}
We again distinguish the two cases:
\begin{itemize}
\item \emph{Either $\gamma\leq-\frac{\omega_1}{2}$ and 
	$\alpha>0$ is arbitrary, or $\gamma>-\frac{\omega_1}{2}$ 
	and $\alpha>\alpha_3$, where}
	\begin{equation}\label{alpha.cr.3}
	\ds\alpha_3:=\alpha_3(\gamma):=
		\frac{\sqrt{\omega_1+2\gamma}}{\bigl\|
			\frac{\la}{\sqrt{\omega_1-\omega}}
			\bigr\|_{L^2(\RR^d)}}.
	\end{equation}
	In this case, we have $\Phi(\omega_1+\gamma)>0$. Hence, 
	the second relation in \eqref{lim.at.pminfty} and the 
	monotonicity imply that the continuous function 
	$\Phi$ has \emph{a unique zero} -- denoted by 
	$F:=F(\gamma, \alpha)$ -- in the interval 
	$(\omega_1+\gamma, \infty)$.
	\item \emph{The case $\gamma>-\frac{\omega_1}{2}$ and 
	$0<\alpha\leq \alpha_3$}. In this case, we have 
	$\Phi(\omega_1+\gamma)\leq0$. Hence, the second relation 
	in \eqref{lim.at.pminfty} and the monotonicity imply that 
	the continuous function $\Phi$ does not vanish in the 
	interval $(\omega_1+\gamma, \infty)$.	
\end{itemize}
If \eqref{strong.ultrared.regularity.M} does not hold, then 
the second limit in \eqref{mon.conv.thm:lim.at.boundary} is 
positive infinity for all $\alpha>0$. Hence, the second 
relation in \eqref{lim.at.pminfty} and the monotonicity imply 
that the continuous function $\Phi$ has \emph{a unique zero} 
(again denoted by) $F:=F(\gamma, \alpha)$ in the interval 
$(\omega_1+\gamma, \infty)$ for all $\alpha>0$.

\smallskip

In addition to \eqref{alpha.cr.1} and \eqref{alpha.cr.3}, we 
will also be dealing with the following particular values 
of the coupling constant
\begin{equation}\label{alpha2,4}
	\ds\alpha_2:=\alpha_2(\gamma):=
	\frac{\sqrt{\omega_1-2\omega_0-2\gamma}}{\Bigl\|
		\frac{\la}{\sqrt{\omega_1-2\omega_0+\omega}}
		\Bigr\|_{L^2(\RR^d)}}, 
	\quad 
	\ds\alpha_4:=\alpha_4(\gamma):=
	\frac{\sqrt{2\omega_1-\omega_0+2\gamma}}{\Bigl\|
		\frac{\la}{\sqrt{2\omega_1-\omega_0-\omega}}
		\Bigr\|_{L^2(\RR^d)}}
\end{equation}
whenever $\gamma\leq\frac{\omega_1}{2}-\omega_0$ and 
$\gamma\geq\frac{\omega_0}{2}-\omega_1$, respectively. 
Since the denominators of both fractions in \eqref{alpha2,4} 
lie between 
$\frac{1}{\sqrt{2(\omega_1-\omega_0)}}\|\la\|_{L^2(\RR^d)}$ 
and $\frac{1}{\sqrt{\omega_1-\omega_0}}\|\la\|_{L^2(\RR^d)}$, 
the quantities in \eqref{alpha2,4} are well-defined. 

\subsection{Essential spectrum of the reduced operator matrix}
%
The following theorem provides an explicit description of the 
essential spectrum of the self-adjoint operator matrix $\cA$ 
corresponding to 
\eqref{reduced.op.mat.}-\eqref{op.entr.op.mat.A}. 
The structure of the essential 
spectrum depends on the location of the parameters $\alpha>0$ 
and $\gamma\in\RR$ as well as the violation of the conditions 
\eqref{strong.ultrared.regularity.m} and 
\eqref{strong.ultrared.regularity.M}.
\begin{theorem}\label{thm:ess.spec.A}
Let the coupling constant $\alpha>0$ be arbitrary. 
\begin{enumerate}[\upshape (i)]
\item Suppose that both of \eqref{strong.ultrared.regularity.m} 
	and \eqref{strong.ultrared.regularity.M} are violated. 
	\begin{enumerate}[\upshape (a)]
	\item If $\gamma \in\bigl(-\infty, 
	\frac{\omega_0}{2}-\omega_1\bigr]$, then 
		\begin{equation*}
		\sess(\cA)=
		\begin{cases} 
		[\omega_0+E,2\omega_1+\gamma]\cup[\omega_0+F,\omega_1+F] 
			& \mbox{if } \alpha\in(0,\alpha_2], \\
		[\omega_0+E,\omega_1+E]\cup[2\omega_0+\gamma,2\omega_1+\gamma]
		\cup[\omega_0+F,\omega_1+F] 
			& \mbox{if } \alpha\in(\alpha_2,\infty).
		\end{cases}
		\end{equation*}
	\item If $\gamma\in\bigl(\frac{\omega_0}{2}-\omega_1, 
	\frac{\omega_1}{2}-\omega_0\bigr]$, then 
		\begin{equation*}
		\sess(\cA)= 
		\begin{cases} 
		[\omega_0+E,\omega_1+F] 
			& \mbox{if } \alpha\in(0,\alpha_2], \\
		[\omega_0+E,\omega_1+E]\cup[2\omega_0+\gamma,\omega_1+F] 
			& \mbox{if } \alpha\in(\alpha_2, \alpha_4], \\
		[\omega_0+E,\omega_1+E]
			\cup[2\omega_0+\gamma,2\omega_1+\gamma]
			\cup[\omega_0+F,\omega_1+F] 
			& \mbox{if } \alpha\in(\alpha_4, \infty).
		\end{cases}
		\end{equation*}
	\item If $\gamma\in\bigl(\frac{\omega_1}{2}-\omega_0,\infty\bigr)$, 
	then 
		\begin{equation*}
		\sess(\cA)= 
		\begin{cases} 
		[\omega_0+E,\omega_1+E]\cup[2\omega_0+\gamma,\omega_1+F] 
			& \mbox{if } \alpha\in(0,\alpha_4],\\
		[\omega_0+E,\omega_1+E]\cup
			[2\omega_0+\gamma,2\omega_1+\gamma]\cup 
			[\omega_0+F,\omega_1+F] 
			& \mbox{if } \alpha\in(\alpha_4, \infty).
		\end{cases}
		\end{equation*}
\end{enumerate}
\item Suppose that \eqref{strong.ultrared.regularity.m} is 
satisfied but \eqref{strong.ultrared.regularity.M} is 
violated.
	\begin{enumerate}[\upshape (a)]
	\item If 
	$\gamma\in\bigl(-\infty,\frac{\omega_0}{2}-\omega_1\bigr]$, 
	then 
		\begin{equation*}
		\sess(\cA)= 
		\begin{cases} 
		[2\omega_0+\gamma,2\omega_1+\gamma]\cup
		[\omega_0+F,\omega_1+F] 
			& \mbox{if } \alpha\in(0,\alpha_1],\\
		[\omega_0+E,2\omega_1+\gamma]\cup 
		[\omega_0+F,\omega_1+F] 
			& \mbox{if } \alpha\in(\alpha_1,\alpha_2],\\
		[\omega_0+E,\omega_1+E]\cup
		[2\omega_0+\gamma,2\omega_1+\gamma]\cup
		[\omega_0+F,\omega_1+F] 
			& \mbox{if } \alpha\in(\alpha_2, \infty).
		\end{cases}
		\end{equation*}
	\item If 
	$\gamma\in\bigl(\frac{\omega_0}{2}-\omega_1,
	-\frac{\omega_0}{2}\bigr]$, then
		\begin{equation*}
		\sess(\cA)= 
		\begin{cases} 
		[2\omega_0+\gamma,\omega_1+F] 
			& \mbox{if } \alpha\in(0,\alpha_1],\\
		[\omega_0+E,\omega_1+F] 
			& \mbox{if } \alpha\in(\alpha_1,\alpha_2],\\
		[\omega_0+E,\omega_1+E]\cup[2\omega_0+\gamma,\omega_1+F] 
			& \mbox{if } \alpha\in(\alpha_2, \alpha_4],\\
		[\omega_0+E,\omega_1+E]\cup
		[2\omega_0+\gamma,2\omega_1+\gamma]\cup
		[\omega_0+F,\omega_1+F] 
			& \mbox{if } \alpha\in(\alpha_4, \infty).
		\end{cases}
		\end{equation*}
	\item If $\gamma\in\bigl(-\frac{\omega_0}{2}, \frac{\omega_1}{2}-\omega_0\bigr]$, then
	\begin{equation*}
	\sess(\cA)= 
	\begin{cases} 
	[\omega_0+E,\omega_1+F] 
	& \mbox{if } \alpha\in(0,\alpha_2],\\
	[\omega_0+E,\omega_1+E]\cup[2\omega_0+\gamma,\omega_1+F] 
	& \mbox{if } \alpha\in(\alpha_2, \alpha_4],\\
	[\omega_0+E,\omega_1+E]\cup
	[2\omega_0+\gamma,2\omega_1+\gamma]\cup
	[\omega_0+F,\omega_1+F] 
	& \mbox{if } \alpha\in(\alpha_4, \infty).
	\end{cases}
	\end{equation*}
	\item If $\gamma\in\bigl(\frac{\omega_1}{2}-\omega_0, \infty)$, 
	then 
		\begin{equation*}
		\sess(\cA)= 
		\begin{cases} 
		[\omega_0+E,\omega_1+E]
		\cup[2\omega_0+\gamma,\omega_1+F] 
			& \mbox{if } \alpha\in(0, \alpha_4],\\
		[\omega_0+E,\omega_1+E]\cup
		[2\omega_0+\gamma,2\omega_1+\gamma]\cup
		[\omega_0+F,\omega_1+F] 
			& \mbox{if } \alpha\in(\alpha_4, \infty).
		\end{cases}
		\end{equation*}
	\end{enumerate}
\item Suppose that \eqref{strong.ultrared.regularity.m} is 
violated but \eqref{strong.ultrared.regularity.M} is 
satisfied.
	\begin{enumerate}[\upshape (a)]
	\item If $\gamma\in\bigl(-\infty,
	\frac{\omega_0}{2}-\omega_1\bigr]$, then
	\begin{equation*}
		\sess(\cA)= 
		\begin{cases} 
		[\omega_0+E,2\omega_1+\gamma]\cup
		[\omega_0+F,\omega_1+F] 
			& \mbox{if } \alpha\in(0,\alpha_2],\\
		[\omega_0+E,\omega_1+E]\cup
		[2\omega_0+\gamma,2\omega_1+\gamma]\cup
		[\omega_0+F,\omega_1+F] 
			& \mbox{if } \alpha\in(\alpha_2, \infty).
		\end{cases}
	\end{equation*}
	\item If $\gamma\in\bigl(\frac{\omega_0}{2}-\omega_1,
	-\frac{\omega_1}{2}\bigr]$, then 
	\begin{equation*}
		\sess(\cA)= 
		\begin{cases} 
		[\omega_0+E,\omega_1+F] 
			& \mbox{if } \alpha\in(0,\alpha_2],\\
		[\omega_0+E,\omega_1+E]\cup
		[2\omega_0+\gamma,\omega_1+F] 
			& \mbox{if } \alpha\in(\alpha_2, \alpha_4],\\
		[\omega_0+E,\omega_1+E]\cup
		[2\omega_0+\gamma,2M+\gamma]
		\cup[\omega_0+F,\omega_1+F] 
			& \mbox{if } \alpha\in(\alpha_4, \infty).
		\end{cases}
	\end{equation*}
	\item If $\gamma\in\bigl(-\frac{\omega_1}{2},\frac{\omega_1}{2}-\omega_0\bigr]$, 
	then
	\begin{equation*}
		\sess(\cA)= 
		\begin{cases} 
		[\omega_0+E,2\omega_1+\gamma] 
			& \mbox{if } \alpha\in(0,\alpha_2],\\
		[\omega_0+E,\omega_1+E]\cup
		[2\omega_0+\gamma,2\omega_1+\gamma] 
			& \mbox{if } \alpha\in(\alpha_2,\alpha_3],\\
		[\omega_0+E,\omega_1+E]\cup
		[2\omega_0+\gamma,\omega_1+F] 
			& \mbox{if } \alpha\in(\alpha_3,\alpha_4],\\
		[\omega_0+E,\omega_1+E]\cup
		[2\omega_0+\gamma,2\omega_1+\gamma]\cup
		[\omega_0+F,\omega_1+F] 
			& \mbox{if } \alpha\in(\alpha_4,\infty).
		\end{cases}
	\end{equation*}
	\item If $\gamma\in\bigl(
	\frac{\omega_1}{2}-\omega_0,\infty)$, then 
	\begin{equation*}
		\sess(\cA)= 
		\begin{cases} 
		[\omega_0+E,\omega_1+E]\cup
		[2\omega_0+\gamma,2\omega_1+\gamma] 
			& \mbox{if } \alpha\in(0,\alpha_3],\\
		[\omega_0+E,\omega_1+E]\cup
		[2\omega_0+\gamma,\omega_1+F] 
			& \mbox{if } \alpha\in(\alpha_3, \alpha_4],\\
		[\omega_0+E,\omega_1+E]\cup
		[2\omega_0+\gamma,2\omega_1+\gamma]\cup
		[\omega_0+F,\omega_1+F] 
			& \mbox{if } \alpha\in(\alpha_4, \infty).
		\end{cases}
	\end{equation*}
	\end{enumerate}
\item Suppose that both of \eqref{strong.ultrared.regularity.m} 
and \eqref{strong.ultrared.regularity.M} are satisfied.
	\begin{enumerate}[\upshape (a)]
	\item If $\gamma\in\bigl(-\infty, 
	\frac{\omega_0}{2}-\omega_1\bigr]$, then 
	\begin{equation*}
		\sess(\cA)= 
		\begin{cases} 
		[2\omega_0+\gamma,2\omega_1+\gamma]\cup
		[\omega_0+F,\omega_1+F] 
			& \mbox{if } \alpha\in(0,\alpha_1],\\
		[\omega_0+E,2\omega_1+\gamma]\cup
		[\omega_0+F,\omega_1+F] 
			& \mbox{if } \alpha\in(\alpha_1,\alpha_2], \\
		[\omega_0+E,\omega_1+E]\cup
		[2\omega_0+\gamma,2\omega_1+\gamma]\cup
		[\omega_0+F,\omega_1+F] 
			& \mbox{if } \alpha\in(\alpha_2, \infty).
		\end{cases}
	\end{equation*}
	\item If $\gamma\in\bigl(\frac{\omega_0}{2}-\omega_1,
	-\frac{\omega_1}{2}\bigr]$, then
	\begin{equation*}
		\sess(\cA)= 
		\begin{cases} 
		[2\omega_0+\gamma, \omega_1+F] 
			& \mbox{if } \alpha\in(0,\alpha_1],\\ 
		[\omega_0+E,\omega_1+F] 
			& \mbox{if } \alpha\in(\alpha_1,\alpha_2],\\
		[\omega_0+E,\omega_1+E]\cup
		[2\omega_0+\gamma,\omega_1+F] 
			& \mbox{if } \alpha\in(\alpha_2, \alpha_4],\\
		[\omega_0+E,\omega_1+E]\cup
		[2\omega_0+\gamma,2\omega_1+\gamma]\cup
		[\omega_0+F,\omega_1+F] 
			& \mbox{if } \alpha\in(\alpha_4,\infty).
		\end{cases}
	\end{equation*}
	\item If $\gamma\in\bigl(-\frac{\omega_1}{2},-\frac{\omega_0}{2}\bigr]$, then 
	\begin{equation*}
		\sess(\cA)= 
		\begin{cases} 
		[2\omega_0+\gamma,2\omega_1+\gamma] 
			& \mbox{if } \alpha\in(0,\alpha_1],\\ 
		[\omega_0+E,2\omega_1+\gamma] 
			& \mbox{if } \alpha\in(\alpha_1,\alpha_2],\\
		[\omega_0+E,\omega_1+E]\cup
		[2\omega_0+\gamma,2\omega_1+\gamma] 
			& \mbox{if } \alpha\in(\alpha_2, \alpha_3],\\
		[\omega_0+E,\omega_1+E]\cup
		[2\omega_0+\gamma,\omega_1+F] 
			& \mbox{if } \alpha\in(\alpha_3, \alpha_4],\\
		[\omega_0+E,\omega_1+E]\cup
		[2\omega_0+\gamma,2\omega_1+\gamma]\cup
		[\omega_0+F,\omega_1+F] 
			& \mbox{if } \alpha\in(\alpha_4,\infty).
		\end{cases}
	\end{equation*}
	\item If $\gamma\in\bigl(-\frac{\omega_0}{2}, 
	\frac{\omega_1}{2}-\omega_0\bigr]$, then 
	\begin{equation*}
		\sess(\cA)= 
		\begin{cases} 
		[\omega_0+E,2\omega_1+\gamma] 
			& \mbox{if } \alpha\in(0,\alpha_2],\\ 
		[\omega_0+E,\omega_1+E]\cup
		[2\omega_0+\gamma,2\omega_1+\gamma] 
			& \mbox{if } \alpha\in(\alpha_2,\alpha_3],\\
		[\omega_0+E,\omega_1+E]\cup
		[2\omega_0+\gamma,\omega_1+F] 
			& \mbox{if } \alpha\in(\alpha_3, \alpha_4],\\
		[\omega_0+E,\omega_1+E]\cup
		[2\omega_0+\gamma,2\omega_1+\gamma]\cup
		[\omega_0+F,\omega_1+F] & \mbox{if } \alpha\in(\alpha_4, \infty).
		\end{cases}
	\end{equation*}
	\item If $\gamma\in\bigl(\frac{\omega_1}{2}-\omega_0, 
	\infty\bigr)$, then 
	\begin{equation*}
		\sess(\cA)= 
		\begin{cases} 
		[\omega_0+E,\omega_1+E]\cup
		[2\omega_0+\gamma,2\omega_1+\gamma] 
			& \mbox{if } \alpha\in(0,\alpha_3],\\
		[\omega_0+E,\omega_1+E]\cup
		[2\omega_0+\gamma,\omega_1+F] 
			& \mbox{if } \alpha\in(\alpha_3, \alpha_4],\\
		[\omega_0+E,\omega_1+E]\cup
		[2\omega_0+\gamma,2\omega_1+\gamma]\cup
		[\omega_0+F,\omega_1+F] 
			& \mbox{if } \alpha\in(\alpha_4,\infty).
		\end{cases}
	\end{equation*}
	\end{enumerate}
\end{enumerate}
\end{theorem}
\begin{proof}
Let 
$z_0\in(2\omega_0+\gamma,2\omega_1+\gamma)$ be 
arbitrary. Then $z_0=2\omega(k_0)+\gamma$ for some 
$k_0\in\RR^d$. Further, let $\varphi\in C^\infty(\RR^d)$ be 
an arbitrary function supported in the annulus 
$\{k\in \RR^d : 0.5\leq\|k\|<1\}$ and such that 
$\|\varphi\|_{L^2(\RR^d)}=1$. It follows that the orthonormal 
system $\{\Psi_n\}_{n\in\NN}:=\{(0,\psi_n)^t\}_{n\in\NN}$, 
where 
\begin{equation}
\psi_n(k,q):=2^{nd+\frac{1}{2}}\varphi(2^n(k-k_0))
	\varphi(2^n(q-k_0)), \quad k,q\in\RR^d, 
\end{equation}
is a singular sequence for the operator matrix $\cA-z_0$. 
Hence, as in \cite{Ibrogimov-AHP2018}, we get the inclusion 
\begin{equation}
	[2\omega_0+\gamma,2\omega_1+\gamma]\subset\sess(\cA).
\end{equation}
Since $[2\omega_0+\gamma,2\omega_1+\gamma]=\sigma(D)$, we can 
use the Frobenius-Schur factorization \eqref{Frobenius-Schur} 
to describe $\sess(\cA)\cap\rho(D)$, see~\eqref{Schur.power.ess.spec}.
In fact, as in the proof of \cite[Theorem~1]{Ibrogimov-AHP2018}, 
a simple perturbation argument yields
\begin{equation}\label{ess.spec.thm.2.abstract}
\ds\sess(\cA)=[2\omega_0+\gamma,2\omega_1+\gamma]\cup
	\bigl\{z\notin[2\omega_0+\gamma,2\omega_1+\gamma]: 
		0\in\sess(\Delta(z))\bigr\}.
\end{equation}
Since $\Phi$ is strictly decreasing on its domain of 
definition, we have 
\begin{equation}
	\inf_{k\in\RR^d}\Delta(k;z)=
	\Phi(z-\omega_0), 
	\quad 
	\sup_{k\in\RR^d}\Delta(k;z)=
	\Phi(z-\omega_1)
\end{equation}
for all 
$z\in(-\infty,2\omega_0+\gamma)\cup(2\omega_1+\gamma,\infty)$. 
This and the continuity of $\Delta(\mydot,\mydot)$ in the 
second variable imply that 
\begin{equation}\label{sess.spec.A.bdd.in.proof}
	\ds\sess(\cA)=\Sigma_1\cup
	[2\omega_0+\gamma,2\omega_1+\gamma]\cup\Sigma_2,
\end{equation}
where 
\begin{equation}
	\Sigma_1:=\bigl\{z<2\omega_0+\gamma: 
	\Phi(z-\omega_0)\leq 0 \leq \Phi(z-\omega_1)\bigr\}
\end{equation}
and
\begin{equation}
	\Sigma_2:=\bigl\{z>2\omega_1+\gamma: 
	\Phi(z-\omega_0)\leq 0 \leq \Phi(z-\omega_1)\bigr\}.
\end{equation}
On the other hand, by the monotonicity of the continuous 
function $\Phi$, we easily obtain 
\begin{equation}
	\Sigma_1=
	\begin{cases} 
	\varnothing & \mbox{if } \Phi(\omega_0+\gamma)\geq 0,\\ 
	[\omega_0+E, 2\omega_0+\gamma) 
		& \mbox{if } \Phi(\omega_0+\gamma)<0 
		\mbox{ and } \Phi(2\omega_0-\omega_1+\gamma)\geq0,\\
	[\omega_0+E,\omega_1+E] 
		& \mbox{if } \Phi(2\omega_0-\omega_1+\gamma)<0
	\end{cases}
\end{equation} 
and, similarly,
\begin{equation}
	\Sigma_2=
	\begin{cases} 
	\varnothing & \mbox{if } \Phi(\omega_1+\gamma)\leq 0,\\ 
	(2\omega_1+\gamma,\omega_1+F] 
		& \mbox{if } \Phi(\omega_1+\gamma)>0 
		\mbox{ and } \Phi(2\omega_1-\omega_0+\gamma)\leq0,\\
	[\omega_0+F,\omega_1+F] 
		& \mbox{if } \Phi(2\omega_1-\omega_0+\gamma)>0.
	\end{cases}
\end{equation} 
Now elementary analyses of the signs of 
$\Phi(\omega_0+\gamma)$, $\Phi(\omega_1+\gamma)$, 
$\Phi(2\omega_0-\omega_1+\gamma)$ and 
$\Phi(2\omega_1-\omega_0+\gamma)$ in terms of the parameters 
$\alpha$ and $\gamma$ together with \eqref{sess.spec.A.bdd.in.proof} 
complete the proof. 
\end{proof}
%
%
%
%
%
%
\begin{theorem}\label{thm:asymp.length.og.gaps} 
For each fixed $\gamma\in\RR$, the following asymptotic 
expansions hold
\begin{equation}\label{asymp.exp.of.E1.E2}
	\begin{aligned}
	\ds E(\alpha) &=2\omega_0-\omega_1+\gamma-
	\frac{2(\omega_1-2\omega_0-2\gamma)}{\alpha_2+\alpha_2^3
		\bigl\|\frac{\la}{\omega+\omega_1-2\omega_0}
			\bigr\|^2_{L^2(\RR^d)}}(\alpha-\alpha_2)
			+\rm{o}(\alpha-\alpha_2), & \alpha\downarrow\alpha_2,
	\\[1ex]
	\ds F(\alpha) &= 2\omega_1-\omega_0+\gamma+
	\frac{2(2\omega_1-\omega_0+2\gamma)}{\alpha_4+\alpha_4^3
		\bigl\|\frac{\la}{2\omega_1-\omega_0-\omega}
			\bigr\|^2_{L^2(\RR^d)}}(\alpha-\alpha_4)
			+\rm{o}(\alpha-\alpha_4),  & \alpha\uparrow\alpha_4,
	\end{aligned}
\end{equation}
provided that $\gamma\leq\frac{\omega_1}{2}-\omega_0$ and 
$\gamma\geq\frac{\omega_0}{2}-\omega_1$, respectively.
\end{theorem}
\begin{proof}
First of all, we note that the quantities involved in 
the denominators of the second terms in the asymptotic 
expansions \eqref{asymp.exp.of.E1.E2} are well-defined 
as a consequence of the simple estimate 
\begin{equation*}
	\max\Bigl\{\Bigl\|\frac{\la}{\omega+\omega_1-2\omega_0}
		\Bigr\|_{L^2(\RR^d)}, \; 
		\Bigl\|\frac{\la}{2\omega_1-\omega_0-\omega}
		\Bigr\|_{L^2(\RR^d)}\Bigr\}\leq 
		\bigl(\omega_1-\omega_0\bigr)^{-1}\|\la\|_{L^2(\RR^d)}.
\end{equation*}
We derive the first asymptotic expansion only. The second one 
can be derived similarly. To this end, let us consider the 
function 
\begin{equation}
	\psi(x,y)=-\gamma-y-x^2\int \frac{|\la(q)|^2\,\d q}
		{\omega(q)+\gamma-y}
\end{equation}
for $(x,y)\in [\alpha_2,\infty)\times(-\infty,\omega_0+\gamma]$. 
We have $\psi(\alpha_2, 2\omega_0-\omega_1+\gamma)=0$ and
\begin{equation}
	\frac{\partial\psi}{\partial y}
		(\alpha_2, 2\omega_0-\omega_1+\gamma)=
		-1-\alpha^2_2\int \frac{|\la(q)|^2\,\d q}
			{(\omega(q)-2\omega_0+\omega_1)^2}\leq-1.
\end{equation}
Hence, the implicit function theorem yields the existence 
of a constant $\delta>0$ and a unique continuously 
differentiable function 
$\wh E\colon[\alpha_2, \alpha_2+\delta)\to\RR$ such that 
\begin{equation}\label{E(alpha2)}
	\wh E(\alpha_2)=2\omega_0-\omega_1+\gamma 
\end{equation}
and $\psi(\alpha, \wh E(\alpha))=0$ for all 
$\alpha\in[\alpha_2, \alpha_2+\delta)$. Moreover, we have
\begin{equation}\label{implicit.derivative}
	\ds \wh E'(\alpha_2+0)=
	%
	%
	-\frac{2(\omega_1-2\omega_0-2\gamma)}{\alpha_2+\alpha_2^3
		\bigl\|\frac{\la}{\omega+\omega_1-2\omega_0}
			\bigr\|_{L^2(\RR^d)}^2}.
\end{equation} 
On the other hand, $\Phi(\wh E(\alpha))=0$ for all 
$\alpha\in[\alpha_2, \alpha_2+\delta)$ and, 
since $E(\alpha)$ is the unique zero of the function 
$\Phi$ in the interval $(-\infty,\omega_0+\gamma)$ for all 
$\alpha\geq\alpha_2$, we conclude that $E(\alpha)=\wh E(\alpha)$ 
for all $\alpha\in[\alpha_2, \alpha_2+\delta)$. This implies 
that $E\colon [\alpha_2, \alpha_2+\delta)\to\RR$ is continuously 
differentiable and the desired asymptotic expansion 
follows from \eqref{E(alpha2)} and \eqref{implicit.derivative}. 
\end{proof}
%
%
%
\subsection{Discrete spectrum of the reduced operator matrix}
%
%
\begin{theorem}\label{thm:disc.spec.1}
Let the coupling constant $\alpha>0$ be arbitrary. Then none 
of $\omega_0+E$, $\omega_1+E$, $\omega_0+F$ and $\omega_1+F$ 
can be an accumulation point of the discrete spectrum of the 
operator matrix $\cA$. 
\end{theorem}
\begin{proof}
The proof in the case when $\omega_0+E$ is the bottom 
(resp.~$\omega_1+F$ is the top) of the essential spectrum is 
completely analogous to that of 
\cite[Theorem~1ii)]{Ibrogimov-AHP2018}. Here we give a proof 
of the absence of the discrete eigenvalue accumulation at 
$\omega_1+E$ from the right only. The proof of of the absence 
of the discrete eigenvalue accumulation at $m+F$ from the left 
is analogous. To this end, let $\omega_1+E$ be an edge of the 
essential spectrum so that $(\omega_1+E, 2\omega_0+\gamma)$ 
is an essential spectral gap (see 
Theorem~\ref{thm:ess.spec.A}). 
Let $z\in[\omega_1+E,2\omega_0+\gamma)$ be fixed for a moment. 
Since $[\omega_1+E,2\omega_0+\gamma)\subset\rho(D)$, 
the integral operator $K(z)\colon L^2(\RR^d)\to L^2(\RR^d)$ is 
well-defined and compact (see Remark~\ref{rem:on.Delta.and.K}). 
Consider the decomposition
\begin{equation}
	\frac{1}{\omega(k_1)+\omega(k_2)+\gamma-z}=
	\Psi_1(k_1,k_2;z)+\Psi_2(k_1,k_2;z),
\end{equation}
where
\begin{equation}\label{kernel:k_1} 								
	\Psi_1(k_1,k_2;z):=
		\frac{1}{\omega(k_1)+\omega_1+\gamma-z}+
		\frac{1}{\omega(k_2)+\omega_1+\gamma-z}-
		\frac{1}{2\omega_1+\gamma-z}
\end{equation}
and
\begin{equation}\label{kernel:k_2}
	\Psi_2(k_1,k_2;z):=
		\frac{1}{\omega(k_1)+\omega(k_2)+\gamma-z}-
		\Psi_1(k_1,k_2;z),\quad k_1,k_2\in\RR^d.
\end{equation}
Let us denote by $K_1(z)$ and $K_2(z)$ the integral operators 
in $L^2(\RR^d)$ whose kernels are respectively the functions 
$(k_1,k_2)\mapsto\alpha^2\la(k_1)\ov{\la(k_2)}\Psi_1(k_1,k_2;z)$ 
and 
$(k_1,k_2)\mapsto\alpha^2\la(k_1)\ov{\la(k_2)}\Psi_2(k_1,k_2;z)$ 
so that we have the corresponding decomposition
\begin{equation}\label{decomposition.of.K}
K(z)=K_1(z)+K_2(z).
\end{equation} 
Since each summand in \eqref{kernel:k_1} is uniformly bounded 
by the constant $\frac{1}{2\omega_0+\gamma-z}$, we infer from 
the first condition in \eqref{cond:la.la.omega.in.L2} that 
$K_1(z)$ is a well-defined rank-two operator in $L^2(\RR^d)$. 
In fact, the range of $K_1(z)$ coincides with the subspace of 
$L^2(\RR^d)$ spanned by the functions $\la$ and 
$\frac{\la}{\omega+\omega_1+\gamma-z}$. Therefore, using 
\eqref{Schur2} and \cite[Theorem~IX.3.3]{Birman-Solomjak-87b}, 
we obtain   
\begin{equation}\label{lem:estimates.EV.cf.1001}
	\begin{aligned}
	N\bigl((-\infty,0); -S(z)\bigr)
	&=N\bigl((-\infty,0); -\Delta(z)+K(z)\bigr)\\
	&=N\bigl((-\infty,0); -\Delta(z)+K_1(z)+K_2(z)\bigr)\\
	&\leq N\bigl((-\infty,0); -\Delta(z)+K_2(z)\bigr)+2.
	\end{aligned} 
\end{equation}
Hence, in view of Lemma~\ref{lem:EV.count.1}, it suffices to 
show that 
\begin{equation}\label{lem:estimates.EV.cf.1002}
	N\bigl((-\infty,0); -\Delta(\omega_1+E)+
		K_2(\omega_1+E)\bigr)
		<\infty.
\end{equation}
Let $\Omega$ denote the complement of the level set of $\omega$ 
corresponding to $\omega_1$, that is, 
$\Omega:=\{k\in\RR^d: \omega(k)\neq \omega_1\}$.
In view of our hypothesis on the dispersion relation $\omega$, 
it is clear that $\Omega$ is an open subset of $\RR^d$ with 
positive Lebesgue measure. For 
$z\in[\omega_1+E,2\omega_0+\gamma)$, we denote 
by $\Delta_{\Omega}(z)$ and $K_{2,\Omega}(z)$ the restrictions 
of the operators $\Delta(z)$ and $K_{2}(z)$ to $L^2(\Omega)$, 
respectively. Further, we denote by $\Delta_{\Omega}(\mydot;z)$ 
the restriction of the function $\Delta(\mydot;z)$ to $\Omega$ 
so that $\Delta_{\Omega}(z)$ is the operator of multiplication 
by the function $\Delta_{\Omega}(\mydot;z)$ in $L^2(\Omega)$. 
Since $E$ is a zero of $\Phi$, by expressing 
$-\Delta(k;\omega_1+E)$ as $-\Delta(k;\omega_1+E)+\Phi(E)$ and 
doing some elementary calculations, we get the identity
\begin{equation*}
	-\Delta(k;\omega_1+E)=(\omega_1-\omega(k))
	\Bigg(1+\alpha^2\int\frac{|\la(q)|^2\,\d q}
	{(\omega(q)+\gamma-E)(\omega(k)+\omega(q)+\gamma-\omega_1-E)}
	\Bigg).
\end{equation*}
This implies that 
\begin{equation}\label{Delta.sim.omega}
	-\Delta_{\Omega}(k;\omega_1+E)\geq \omega_1-\omega(k)>0 
	\quad \text{for all} \quad k\in\Omega
\end{equation}
and that $\Delta(\mydot;\omega_1+E)\equiv0$ on the level set 
of $\omega$ corresponding to $\omega_1$. So the restriction 
of $\Delta(\omega_1+E)$ to $L^2(\RR^d\setminus\Omega)$ is the 
zero operator. Moreover, it is easy to see from 
\eqref{kernel:k_1}-\eqref{kernel:k_2} that the restriction of 
$K_2(\omega_1+E)$ to $L^2(\RR^d\setminus\Omega)$ is the zero 
operator, too. That is why \eqref{lem:estimates.EV.cf.1002} 
is equivalent to 
\begin{equation}\label{lem:estimates.EV.cf.1012}
	N\bigl((-\infty,0); -\Delta_{\Omega}(\omega_1+E)
		+K_{2,\Omega}(\omega_1+E)\bigr)<\infty.
\end{equation}
To show the latter, first we recall the monotonicity of 
the function $\Phi$ in the interval 
$(\omega_1+E, 2\omega_0+\gamma)$. For all 
$z\in (\omega_1+E, 2\omega_0+\gamma)$, this implies that
\begin{equation}
	\sup_{k\in\Omega}\Delta_{\Omega}(k;z)\leq
	\sup_{k\in\RR^d}\Delta(k;z)=\Phi(z-\omega_1)<\Phi(E)=0.
\end{equation}
Hence, the multiplication operator $-\Delta_{\Omega}(z)$ is 
positive for all $z\in(\omega_1+E,2\omega_0+\gamma)$. Since 
the integral operator $K_{2,\Omega}(z):L^2(\Omega)\to L^2(\Omega)$ 
is well-defined and Hilbert-Schmidt, it follows that 
the operator 
\begin{equation}\label{B.Sch.op.T.eps}
	T(z):=\bigl(-\Delta_{\Omega}(z)\bigr)^{-1/2}
		K_{2,\Omega}(z)
		\bigl(-\Delta_{\Omega}(z)\bigr)^{-1/2}
\end{equation}
is also well-defined and Hilbert-Schmidt for all 
$z\in(\omega_1+E,2\omega_0+\gamma)$. 

Next, let us define $T(\omega_1+E)$ to be the integral operator 
in $L^2(\Omega)$ with the kernel
\begin{equation}\label{kern.sing.E.1}
	\Theta(k_1,k_2):=
		\frac{\alpha^2\ov{\la(k_1)}\la(k_2)
			\Psi_2(k_1,k_2;\omega_1+E)}
		{\sqrt{-\Delta_{\Omega}(k_1;\omega_1+E)}
			\sqrt{-\Delta_{\Omega}(k_2;\omega_1+E)}
		}, \quad k_1,k_2\in\Omega.
\end{equation}
Let $k_1, k_2\in\Omega$ be fixed for a moment and let 
$a=\omega(k_1)-\omega_1$, $b=\omega(k_2)-\omega_1$ and 
$c=\omega_1+\gamma-E$. Then it follows that 
$0<a+b+c\leq 2c$, $\min\{a+c, b+c\}\geq \omega_0+\gamma-E$, 
$a+b+c\geq 2\omega_0+\gamma-\omega_1-E$ and 
$\sqrt{ab}\leq-\frac{a+b}{2}\leq \omega_1-\omega_0$. 
Therefore, we get
\begin{equation}\label{p-wise.estimate.on.Psi2}
	\begin{aligned}
	0\leq\Psi_2(k_1,k_2;\omega_1+E)
	&=\frac{1}{a+b+c}-\frac{1}{a+c}-\frac{1}{b+c}+\frac{1}{c}\\
	&=\frac{ab(a+b+2c)}{c(a+c)(b+c)(a+b+c)}\\
	&\leq \frac{\sqrt{ab}(-a-b)}{(a+c)(b+c)(a+b+c)}\\
	&\leq\frac{2(\omega_1-\omega_0)\sqrt{\omega_1-\omega(k_1)}
		\sqrt{\omega_1-\omega(k_2)}}
			{(\omega_0+\gamma-E)^2(2\omega_0+\gamma-\omega_1-E)}.
	\end{aligned}
\end{equation}
Combining \eqref{p-wise.estimate.on.Psi2} and 
\eqref{Delta.sim.omega}, we thus obtain the pointwise estimate
\begin{equation}
	|\Theta(k_1,k_2)| \leq 
		\frac{2(\omega_1-\omega_0)|\la(k_1)||\la(k_2)|}
		{(\omega_0+\gamma-E)^2(2\omega_0+\gamma-\omega_1-E)}, 
		\quad k_1,k_2\in\Omega.
\end{equation}
Thus we infer from the first condition in 
\eqref{cond:la.la.omega.in.L2} that 
$\Theta\in L^2(\Omega\times\Omega)$. 
Therefore, $T(\omega_1+E):L^2(\Omega)\to L^2(\Omega)$ 
is a Hilbert-Schmidt operator and the Lebesgue's dominated 
convergence theorem guarantees the right-continuity 
(with respect to the operator norm) of the operator 
function $T$ at $\omega_1+E$. This observation and the 
Weyl inequality (\cf~\cite[Chapter~IX]{Birman-Solomjak-87b})
\begin{equation}
	N\bigl((-\infty,-1);T(z)\bigr)\leq 
	N\bigl((-\infty, -0.5); T(\omega_1+E)\bigr)
	+N\bigl((-\infty, -0.5);-T(\omega_1+E)+T(z)\bigr),
\end{equation}
together with the right-continuity 
(with respect to the operator norm) of the operator 
functions $\Delta_{\Omega}$ and $K_{2,\Omega}$ 
yield 
\begin{equation}
	\begin{aligned}
	N\bigl((-\infty, 0);-\Delta_{\Omega}(\omega_1+E)
		+K_{2, \Omega}(\omega_1+E)\bigr)
	&=\lim_{z\downarrow \omega_1+E}
		N\bigl((-\infty, 0); -\Delta_{\Omega}(z)+K_{2,\Omega}(z)\bigr)\\
	&=\lim_{z\downarrow\omega_1+E}N\bigl((-\infty,-1);T(z)\bigr)\\
	&\leq N\bigl((-\infty,-0.5); T(\omega_1+E)\bigr).
	\end{aligned}
\end{equation}
However, $N\bigl((-\infty,-0.5); T(\omega_1+E)\bigr)$ must be a finite 
number because of the compactness of the operator 
$T(\omega_1+E)$, thus justifying 
\eqref{lem:estimates.EV.cf.1012}.
\end{proof}
%
%
%
%
\begin{theorem}\label{thm:disc.spec.2}
Let the coupling constant $\alpha>0$ be arbitrary. 
\begin{enumerate}[\upshape i)]
\item If $2\omega_0+\gamma$ (resp.~$2\omega_1+\gamma$) is the 
bottom (resp.~top) of the essential spectrum, then it can not 
be an accumulation point of the discrete spectrum of the 
operator matrix $\cA$ provided that 
\begin{equation}\label{cond.strong.on.la.and.bounded.omega.M}
	\frac{\la}{\omega-\omega_0}\in L^2(\RR^d) \quad 
	\bigg(resp.\;\frac{\la}{\omega_1-\omega}\in L^2(\RR^d)\bigg).
\end{equation} 	
\item If $2\omega_0+\gamma$ (resp.~$2\omega_1+\gamma$) is an 
edge but not bottom (resp.~top) of the essential spectrum, 
then it can not be an accumulation point of the discrete 
spectrum of the operator matrix $\cA$ provided that
\begin{equation}\label{cond.strong.on.la.and.bounded.omega.m.M}
	\frac{\la}{\sqrt{(\omega-\omega_0)(\omega_1-\omega)}}
		\in L^2(\RR^d).
\end{equation}	 
\end{enumerate}
\end{theorem}
\begin{proof}
i) We give a proof of the claim about the absence of the 
discrete eigenvalue accumulation at $2\omega_1+\gamma$ from 
the right. The other claim can be justified analogously. 
Let the parameters $\alpha$ and $\gamma$ be such that 
$2\omega_1+\gamma$ is the top of the essential spectrum. 
Then the second condition in 
\eqref{cond.strong.on.la.and.bounded.omega.M} and the 
Cauchy-Schwarz inequality guarantees that 
$\Phi(\omega_1+\gamma)$ is finite, see \eqref{ineq.CS.infr}. 
Since $2\omega_1+\gamma$ is the top of the essential 
spectrum, we have $\Phi(\omega_1+\gamma)\leq0$. Hence, we 
easily obtain
\begin{equation}
	\begin{aligned}
	\Delta(k;2\omega_1+\gamma) 
	&\leq\Delta(k;2\omega_1+\gamma)-\Phi(\omega_1+\gamma)\\
	&=(\omega(k)-\omega_1)
	\Bigg(1+\alpha^2\int\frac{|\la(q)|^2\,\d q}
	{(\omega_1-\omega(q))(2\omega_1-\omega(k)-\omega(q))}
		\Bigg),
	\end{aligned}
\end{equation}
and consequently, 
\begin{equation}\label{Delta=omega.2M}
	-\Delta(k;2\omega_1+\gamma)\geq\omega_1-\omega(k), 
		\quad k\in\RR^d.
\end{equation}
By the monotonicity of the function $\Phi$ in the interval 
$(2\omega_1+\gamma,\infty)$, we get
\begin{equation}
	\sup_{k\in\RR^d}\Delta(k;z)=\Phi(z-\omega_1)
		<\Phi(\omega_1+\gamma)\leq0
\end{equation}
for all $z\in (2\omega_1+\gamma, \infty)$. Hence, the 
multiplication operator 
$-\Delta(z)\colon L^2(\RR^d)\to L^2(\RR^d)$ is positive 
for all $z\in (2\omega_1+\gamma, \infty)$. Since the integral 
operator $K(z)\colon L^2(\RR^d)\to L^2(\RR^d)$ is well-defined 
and Hilbert-Schmidt, it follows that the operator 
\begin{equation}\label{B.Sch.op.T.2M}
	T(z):=\bigl(-\Delta(z)\bigr)^{-1/2}
		K(z)\bigl(-\Delta(z)\bigr)^{-1/2}
\end{equation}
is also well-defined and Hilbert-Schmidt for all 
$z\in(2\omega_1+\gamma, \infty)$. 

\smallskip
\noindent
Next, we define $T(2\omega_1+\gamma)$ to be the integral 
operator in $L^2(\RR^d)$ with the kernel
\begin{equation}\label{kern.sing.2M}
	\Theta(k_1,k_2):= \frac{\alpha^2\ov{\la(k_1)}\la(k_2)}
	{\sqrt{-\Delta(k_1;2\omega_1+\gamma)}
		(\omega(k_1)+\omega(k_2)-2\omega_1)
		\sqrt{-\Delta(k_2;2\omega_1+\gamma)}}.
\end{equation}
Using \eqref{Delta=omega.2M} and the elementary 
mean-inequality
\begin{equation}\label{ineq:AM-GM}
	\frac{1}{2}\bigl(2\omega_1-\omega(k_1)-\omega(k_2)\bigr) 
	\geq \sqrt{(\omega_1-\omega(k_1))(\omega_1-\omega(k_2))}, 
\end{equation} 
we can estimate the kernel in \eqref{kern.sing.2M} as follows
\begin{equation}
	|\Theta(k_1,k_2)|\leq\frac{\alpha^2}{2}
	\frac{|\la(k_1)|}{\omega_1-\omega(k_1)}
	\frac{|\la(k_2)|}{\omega_1-\omega(k_2)},
	\quad k_1,k_2\in\RR^d.
\end{equation}
In view of the second condition in 
\eqref{cond.strong.on.la.and.bounded.omega.M}, we thus conclude 
that $\Theta\in L^2(\RR^d\times\RR^d)$ with
\begin{equation}
	\ds\|\Theta\|_{L^2(\RR^d\times\RR^d)}\leq 
	\frac{\alpha^2}{2}\Bigl\|\frac{\la}{\omega_1-\omega}
		\Bigr\|^2_{L^2(\RR^d)}<\infty.
\end{equation}
Therefore, $T(2\omega_1+\gamma)\colon L^2(\RR^d)\to L^2(\RR^d)$ 
is a Hilbert-Schmidt operator and the Lebesgue's dominated 
convergence theorem guarantees the right-continuity 
(with respect to the operator norm) of the operator function 
$T$ at $2\omega_1+\gamma$. This observation and the Weyl 
inequality (see \cite[Chapter~IX]{Birman-Solomjak-87b}) 
\begin{equation}
	N\bigl((-\infty,-1);T(z)\bigr)
		\leq N\bigl((-\infty,-0.5);T(2\omega_1+\gamma)\bigr)
			+N\bigl((-\infty,-0.5);-T(2\omega_1+\gamma)+T(z)\bigr)
\end{equation}
together with the right-continuity 
(with respect to the operator norm) of the operator 
functions $\Delta$ and $K$ yield 
\begin{equation}
	\begin{aligned}
	N\bigl((-\infty,0); -S(2\omega_1+\gamma)\bigr)
	&=N\bigl((-\infty,0);-\Delta(2\omega_1+\gamma)
		+K(2\omega_1+\gamma)\bigr)\\
	&=\lim_{z\downarrow 2\omega_1
		+\gamma}N\bigl((-\infty,0);-\Delta(z)+K(z)\bigr)\\
	&=\lim_{z\downarrow 2\omega_1+\gamma}N\bigl((-\infty,-1);T(z)\bigr)\\
	&\leq N\bigl((-\infty,-0.5);T(2\omega_1+\gamma)\bigr).
	\end{aligned}
\end{equation}
However, $N\bigl((-\infty,-0.5);T(2\omega_1+\gamma)\bigr)$ must be 
a finite number because of the compactness of the operator 
$T(2\omega_1+\gamma)$. Now Lemma~\ref{lem:EV.count.1} finishes 
the proof.

\medskip
\noindent
ii) Here we give a sketch of the proof of the claim 
about the absence of the discrete eigenvalue accumulation 
at $2\omega_1+\gamma$ from the right. The other claim can 
be justified analogously. Let the parameters $\alpha$ and 
$\gamma$ be such that $2\omega_1+\gamma$ is an edge but 
not the top of the essential spectrum. Then, we have 
$\Phi(2\omega_1-\omega_0+\gamma)\geq0$ (that is, 
$\alpha\in(0,\alpha_2]$, see Theorem~\ref{thm:ess.spec.A}). 
Hence, we easily obtain
\begin{equation}
	\begin{aligned}
	\Delta(k;2\omega_1+\gamma)&\geq\Delta(k;2\omega_1+\gamma)
		-\Phi(2\omega_1-\omega_0+\gamma)\\
	&=(\omega(k)-\omega_0)\Bigg(
		1+\alpha^2\int\frac{|\la(q)|^2\,\d q}
			{(2\omega_1-\omega_0-\omega(q))
				(2\omega_1-\omega(k)-\omega(q))}\Bigg)
	\end{aligned}
\end{equation}
and consequently, 
\begin{equation}\label{Delta=omega.2M.m}
	\Delta(k;2\omega_1+\gamma) \geq \omega(k)-\omega_0, 
		\quad k\in\RR^d.
\end{equation}
By the monotonicity of the function $\Phi$ in the interval 
$(2\omega_1+\gamma, \omega_0+F)$, we have
\begin{equation}
	\inf_{k\in\RR^d}\Delta(k;z)=\Phi(z-\omega_0)
		>\Phi(F)=0
\end{equation}
for all $z\in(2\omega_1+\gamma, \omega_0+F)$. Hence, the 
multiplication operator 
$\Delta(z)\colon L^2(\RR^d)\to L^2(\RR^d)$ 
is positive for all $z\in(2\omega_1+\gamma,\omega_0+F)$. 
Since the integral operator $K\colon L^2(\RR^d)\to L^2(\RR^d)$ 
is well-defined and Hilbert-Schmidt, it follows that the 
operator 
\begin{equation}\label{B.Sch.op.T.2Mm}
	T(z):=\bigl(\Delta(z)\bigr)^{-1/2}K(z)
		\bigl(\Delta(z)\bigr)^{-1/2}
\end{equation}
is also well-defined and Hilbert-Schmidt for all 
$z\in(2\omega_1+\gamma, \omega_0+F)$. 

\smallskip
\noindent
Next, we define $T(2\omega_1+\gamma)$ to be the integral 
operator in $L^2(\RR^d)$ with the kernel
\begin{equation}\label{kern.sing.2M.m}
	\Theta(k_1,k_2):=\frac{\alpha^2\ov{\la(k_1)}\la(k_2)}
	{\sqrt{\Delta(k_1;2\omega_1+\gamma)}
		(\omega(k_1)+\omega(k_2)-2\omega_1)
		\sqrt{\Delta(k_2;2\omega_1+\gamma)}}.
\end{equation}
Using \eqref{Delta=omega.2M.m} and \eqref{ineq:AM-GM}, we 
get the pointwise estimate 
\begin{equation}
	|\Theta(k_1,k_2)|\leq\frac{\alpha^2}{2}
	\frac{|\la(k_1)|}
	{\sqrt{(\omega(k_1)-\omega_0)(\omega_1-\omega(k_1))}}
	\frac{|\la(k_2)|}
	{\sqrt{(\omega(k_2)-\omega_0)(\omega_1-\omega(k_2))}}.
\end{equation}
It thus follows from the condition 
\eqref{cond.strong.on.la.and.bounded.omega.m.M} that 
$\Theta\in L^2(\RR^d\times\RR^d)$ with
\begin{equation}
	\ds\|\Theta\|_{L^2(\RR^d\times\RR^d)} 
	\leq \frac{\alpha^2}{2}\bigg\|
		\frac{\la}{\sqrt{(\omega-m)(M-\omega)}}
		\bigg\|^2_{L^2(\RR^d)}<\infty.
\end{equation}
Therefore, $T(2\omega_1+\gamma)\colon L^2(\RR^d)\to L^2(\RR^d)$ 
is a Hilbert-Schmidt operator and the Lebesgue's dominated 
convergence theorem guarantees the right-continuity 
(with respect to the operator norm) of the operator 
function $T$ at $2\omega_1+\gamma$. The rest of the proof 
is completely analogous to that of part i) and follows by 
applying Weyl's inequality together with 
Lemma~\ref{lem:EV.count.1}. 
\end{proof}
\begin{remark}
\begin{enumerate}[\upshape (i)]
\item What we really need in the proof of 
Theorem~\ref{thm:disc.spec.2}ii) is the square-integrability of 
\begin{equation}
	(k_1,k_2) \mapsto 
	\frac{\la(k_1)}{\sqrt{\omega(k_1)-\omega_0}}
	\frac{1}{\omega(k_1)+\omega(k_2)-2\omega_1}
	\frac{\la(k_2)}{\sqrt{\omega(k_2)-\omega_0}}
\end{equation}
over $\RR^d\times\RR^d$. This requirement is weaker than the 
condition \eqref{cond.strong.on.la.and.bounded.omega.m.M}.
\item The conditions in 
\eqref{cond.strong.on.la.and.bounded.omega.M} 
are stronger versions of the conditions in 
\eqref{strong.ultrared.regularity.M} and 
\eqref{strong.ultrared.regularity.m}. For instance,  
if the first condition in 
\eqref{cond.strong.on.la.and.bounded.omega.M} holds, then 
the Cauchy-Schwarz inequality yields 
\begin{equation}
	\bigg\|\frac{\la}{\sqrt{\omega-\omega_0}}
		\bigg\|^2_{L^2(\RR^d)}
	\leq\|\la\|_{L^2(\RR^d)}
	\bigg\|\frac{\la}{\omega-\omega_0}
		\bigg\|_{L^2(\RR^d)}<\infty.
\end{equation}
%
\item The way we employ the Weyl inequality in the proofs 
of Theorems~\ref{thm:disc.spec.1} and \ref{thm:disc.spec.2} 
is inspired by \cite{Lakaev-Muminov-2003, Ibrogimov-Tretter-OTAA2018}.
\end{enumerate}
\end{remark}
\newpage
\begin{corollary}\label{cor.disc.spec}
Let the coupling constant $\alpha>0$ be arbitrary. Assume 
that both of the following conditions are satisfied
\begin{equation}\label{the.two.infrared.reg.conditions}
	\frac{\la}{\omega_1-\omega}\in L^2(\RR^d), \quad 
	\frac{\la}{\omega-\omega_0}\in L^2(\RR^d).
\end{equation}
Then the discrete spectrum of the operator matrix $\cA$ 
is finite.
\end{corollary}
\begin{proof}
If \eqref{the.two.infrared.reg.conditions} holds, then the 
Cauchy-Schwarz inequality yields 
\begin{equation*}
	\bigg\|\frac{\la}{\sqrt{(\omega_1-\omega)(\omega-\omega_0)}}
		\bigg\|^2_{L^2(\RR^d)}
	\leq
	\bigg\|\frac{\la}{\omega_1-\omega}\bigg\|_{L^2(\RR^d)}
	\bigg\|\frac{\la}{\omega-\omega_0}\bigg\|_{L^2(\RR^d)}
	<\infty.
\end{equation*}
Hence, Theorems~\ref{thm:disc.spec.1} and \ref{thm:disc.spec.2} 
imply the finiteness of the discrete spectrum of $\cA$ in 
all cases.
\end{proof}
%
%
%
\section{Spectrum of the spin-boson Hamiltonian with two bosons}
%
Since $\omega$ is bounded, the natural domain of the unperturbed 
operator $H_0$ coincides with  $\CC^2\otimes\cF^2_s$.
The first condition in \eqref{cond:la.la.omega.in.L2} implies 
the boundedness of the perturbation $H_\alpha-H_0$ and thus 
the expression for $H_{\alpha}$ given in \eqref{Hamiltonian} 
generates a self-adjoint operator in the Hilbert space 
$\CC^2\otimes\cF_s^2$ with the maximal domain 
(\cf~\cite[Theorem~V.4.3]{Kato}). For notational 
convenience, we denote the corresponding self-adjoint operator 
again by~$H_{\alpha}$ and describe its spectrum in the sequel. 

We denote by $\cA^{(+)}_{\alpha}$ and $\cA^{(-)}_{\alpha}$ the 
operator matrices of the form \eqref{op.entr.op.mat.A} 
corresponding to $\gamma=\eps$ and $\gamma=-\eps$, respectively. 
As a simple 
consequence of Theorems~\ref{thm:ess.spec.A}, \ref{thm:disc.spec.1} 
and \ref{thm:disc.spec.2}, the next result describes the essential 
spectrum as well as the finiteness of the discrete spectrum of the 
energy operator $H_\alpha$.
\begin{theorem}\label{thm:spin-boson-spec}
Let the coupling constant $\alpha>0$ be arbitrary. 
\begin{enumerate}[\upshape (i)]
\item\label{thmSB:part.i} 
Then the essential spectrum of $H_\alpha$ is given by
\begin{equation}\label{form.ess.spec.H}
	\sess(H_{\alpha})=
		\sess(\cA^{(+)}_{\alpha})\cup\sess(\cA^{(-)}_{\alpha}) 
\end{equation}
with $\sess(\cA^{\pm}_{\alpha})$ explicitly determined from 
Theorem~{\upshape{\ref{thm:ess.spec.A}}}. 
\item\label{thmSB:part.ii} 
Every edge of the essential spectrum of $H_\alpha$ other than 
$2\omega_0\pm\eps$ and $2\omega_1\pm\eps$ can not be an 
accumulation point of the discrete spectrum of $H_{\alpha}$. 
\item If $2\omega_0\pm\eps$ (resp.~$2\omega_1\pm\eps$) is the 
bottom (resp.~top) of the essential spectrum of 
$\cA^{(\pm)}_\alpha$, then it can not be an accumulation point 
of the discrete spectrum of $H_{\alpha}$ provided that  
\begin{equation}\label{infrared.reg.conditions}
	\frac{\la}{\omega-\omega_0}\in L^2(\RR^d) \quad 
	\bigg(resp.\;\frac{\la}{\omega_1-\omega}\in L^2(\RR^d)\bigg).
\end{equation}
\item If any of $2\omega_0\pm\eps$ (resp.~of $2\omega_1\pm\eps$) 
is an edge but not the bottom (resp.~top) of the essential 
spectrum of $\cA^{(\pm)}_\alpha$, then it can not be an 
accumulation point of the discrete spectrum 
of $H_{\alpha}$ provided that  
\begin{equation}\label{infrared.reg.type.conditions}
	\frac{\la}{\sqrt{(\omega-\omega_0)(\omega_1-\omega)}}
		\in L^2(\RR^d).
\end{equation}
\end{enumerate}
\end{theorem}
\begin{proof}
Consider the unitary transformation 
$U\colon\CC^2\otimes\cF_s^2\to\cF_s^2\oplus\cF_s^2$, defined by
\begin{equation}
	U\colon
	\begin{pmatrix}
	\begin{pmatrix}
	f_0^{(+)}\\
	f_0^{(-)}
	\end{pmatrix},
	\begin{pmatrix}
	f_1^{(+)}\\
	f_1^{(-)}
	\end{pmatrix}, 
	\begin{pmatrix}
	f_2^{(+)}\\
	f_2^{(-)}
	\end{pmatrix}
	\end{pmatrix}
	\mapsto
	\begin{pmatrix}
	\begin{pmatrix}
	f_0^{(+)}\\
	f_1^{(-)}\\
	f_2^{(+)}
	\end{pmatrix},
	\begin{pmatrix}
	f_0^{(-)}\\
	f_1^{(+)}\\
	f_2^{(-)}
	\end{pmatrix}
	\end{pmatrix}.
\end{equation}
By means of this unitary transformation we can 
block-diagonalize the Hamiltonian $H_{\alpha}$ in 
\eqref{Hamiltonian}. In fact, it is not difficult to check that  
\begin{equation}\label{diagonalization.of.Hamiltonian}
	U^*H_{\alpha}U=
	{\rm{diag}}\{P^*\cA^{(+)}_{\alpha}P, P^*\cA^{(-)}_{\alpha}P\}
		+\text{operator of rank at most 4},
\end{equation}
where $P\colon\cF_s^2\to\cH_1\oplus\cH_2$ is the projection operator 
onto the last two components in the Hilbert space 
$\cF_s^2$, $\cA^{(+)}_{\alpha}$ and $\cA^{(-)}_{\alpha}$ are 
the operator matrices of the form \eqref{op.entr.op.mat.A} with 
$\gamma=\eps$ and $\gamma=-\eps$, respectively. Since the 
essential spectrum  as well as the finiteness of the discrete 
spectrum of self-adjoint operators are invariant with respect 
to finite-rank perturbations 
(\cf~\cite[Chapter~9]{Birman-Solomjak-87b}), it follows from 
\eqref{diagonalization.of.Hamiltonian} that the essential 
spectrum of $H_\alpha$ is given by \eqref{form.ess.spec.H}, 
whereas
\begin{equation}\label{ess.spec.relation}
	\sd(H_{\alpha}) \subset 
		\sd(\cA^{(+)}_{\alpha}) \cup \sd(\cA^{(-)}_{\alpha}).
\end{equation}
The claims on the discrete spectrum immediately follow from 
Theorems~\ref{thm:disc.spec.1} and \ref{thm:disc.spec.2} 
applied for the operator matrices $\cA^{(\pm)}_{\alpha}$.
\end{proof}
\begin{remark}
\begin{enumerate}[\upshape (i)]
\item It is easy to check that none of $2\omega_0\pm\eps$ 
can ever be the bottom of the essential spectrum of $H_{\alpha}$. 
Hence, discrete eigenvalues 
of $H_{\alpha}$ can never accumulate to the bottom of the 
essential spectrum of $H_{\alpha}$ from the left. The top of 
the essential spectrum of $H_{\alpha}$ can never be 
$2\omega_1\pm\eps$ whenever \eqref{strong.ultrared.regularity.M} 
is not satisfied. Consequently, discrete eigenvalues of 
$H_{\alpha}$ can never accumulate to the top of the essential 
spectrum of $H_{\alpha}$ from the right whenever 
\eqref{strong.ultrared.regularity.M} is violated.
\item It follows from Corollary~\ref{cor.disc.spec} that 
the discrete
spectrum of $H_\alpha$ is always finite whenever the infrared 
regularity 
conditions \eqref{infrared.reg.conditions} are satisfied. 
We believe that, at least for certain values of the spatial 
dimension $d\geq1$, conditions of type 
\eqref{infrared.reg.conditions}-\eqref{infrared.reg.type.conditions} 
are necessary to guarantee the finiteness of the discrete 
spectrum near the corresponding edges of the essential spectrum. 
In fact, the violation of a infrared regularity condition is 
equivalent to the fact that the corresponding edge of the 
essential spectrum of the two-boson system is a resonance 
state and in this case we expect an effect analogous to the 
Efimov effect from the spectral theory of the standard 
three-body Schr\"odinger operators, \cf~\cite{Jafaev-1974, Sobolev-CMP93, Tamura-1991, 
	Albeverio-Lakaev-Muminov-2004}.
\end{enumerate}
\end{remark}
\subsection*{Acknowledgments}
Most of the results in this paper were obtained during my stay 
at University College London (UCL). I am grateful to 
Prof.~A.~V.~Sobolev for fruitful discussions and I thank the 
Department of Mathematics at UCL for the kind hospitality. 
I would also like to thank Prof.~H.~Spohn for stimulating 
personal communications on the subject of this paper. 
The financial support of the 
\emph{Swiss National Science Foundation} through the Early 
Postdoc.Mobility grant No.~$168723$ is gratefully acknowledged.
%
%
{\small
	\bibliographystyle{acm}
	\bibliography{spin_boson_literature_20190131}
}
\end{document}